\documentclass{amsart}
\usepackage{amsmath,amscd,amssymb}
\usepackage{amsfonts}
\usepackage[all]{xy}
\usepackage{enumerate}
\numberwithin{equation}{subsection}
\theoremstyle{plain}
\newtheorem{thm}[subsection]{Theorem}
\newtheorem{prop}[subsection]{Proposition}
\newtheorem{lemma}[subsection]{Lemma}
\newtheorem{cor}[subsection]{Corollary}
\theoremstyle{definition}
\newtheorem{defn}[subsection]{Definition}

\newtheorem{cont}[subsection]{Contents}
\newtheorem{ackn}[subsection]{Acknowledgement}
\theoremstyle{remark}

\def\noqed{\renewcommand{\qedsymbol}{}}
\begin{document}
\title{On the Chow groups of certain geometrically rational $5$-folds}
\author{Ambrus P\'al}
\date{December 15, 2013}
\address{Department of Mathematics, 180 Queen's Gate, Imperial College, London, SW7 2AZ, United Kingdom}
\email{a.pal@imperial.ac.uk}
\begin{abstract} We give an explicit regular model for the quadric fibration studied in \cite{Pi1}. As an application we show that this construction furnishes a counterexample for the integral Tate conjecture in any odd characteristic for some sufficiently large finite field. We study the \'etale cohomology of this regular model, and as a consequence we derive that these counterexamples are not torsion.
\end{abstract}
\footnotetext[1]{\it 2000 Mathematics Subject Classification. \rm 14M20, 14C15, 14G13.}
\maketitle

\section{Introduction}

Let $k$ be a field and let $x,y,z$ be the homogeneous coordinates for $\mathbb P^2_k$. Let $H$ be the set of linear forms
$$h=e_xx+e_yy+e_zz$$
where $e_x,e_y,e_z\in\{0,1\}$. Let
$$l_i=b_ix +c_iy+d_iz,\quad b_i,c_i,d_i\in k$$
be linear forms, too, where $i =1,2$. Assume that the characteristic of $k$ is not two. Let $a\in k^*-(k^*)^2$ and let $Q$ be the smooth, projective quadric over the function field $F=k(x,y)=k(\mathbb P^2_k)$ of $\mathbb P^2_k$ defined by the homogeneous equation:
\begin{equation}
x^2_0-ax^2_1-fx^2_2+afx^2_3-g_1g_2x^2_4=0
\end{equation}
in $\mathbb P^4_F$, where
$$f=\frac{x}{y},\quad g_1=\frac{\prod_{h\in H}(l_1+h)}{y^8},\quad g_2=\frac{\prod_{h\in H}(l_2+h)}{z^8}.$$
Recall that we say that a set $S$ of lines in $\mathbb P^2_k$ is in general position if the intersection of any three elements of $S$ is empty. Let $L$ be the set:
$$L=\{x=0\}\cup\{y=0\}\cup\{l_1+h=0|h\in H\}\cup\{l_2+h=0|h\in H\}$$
of lines (which means that we assume that none of the linear forms above are zero). Our first result is the following 
\begin{thm} Assume that $L$ is in general position. Then there is a projective smooth variety $X$ defined over $k$ which has a fibration over $\mathbb P^2_k$ whose generic fibre is the quadric $Q$.
\end{thm}
The assumption is not very restrictive; we will show that there is a choice of $l_1$ and $l_2$ such that $L$ is in general position when the cardinality of $k$ is sufficiently large (see Proposition 2.2). The quadric $Q$ has a $k(\sqrt a)$-rational point, and hence it is $k(\sqrt a,x,y)$-rational. Therefore $X$ is rational over $k(\sqrt a)$.   
\begin{thm} Assume that $k$ is a finite field. Also suppose that $L$ is in general position and $b_i,c_i,d_i\neq0,-1$ for $i =1,2$. Then both the natural forgetful map:
\begin{equation}
CH^2(X)\rightarrow CH^2(X\times_k\overline k)^{\mathrm{Gal}(\overline k/k)}
\end{equation}
and the cycle class map:
\begin{equation}
CH^2(X)\otimes\mathbb Z_2\rightarrow H^4_{\textrm{\rm\'et}}(X,\mathbb Z_2(2))
\end{equation}
are not surjective.
\end{thm}
This theorem is an improvement of the main result of the paper \cite{Pi1} in two ways. First we construct a geometrically rational smooth projective variety such that the map in (1.2.1) is not surjective in any odd characteristic for some sufficiently large finite field. Moreover we also show that these varieties are counterexamples to the integral Tate conjecture, too. However the result above leaves the following question unanswered: are the cohomology classes in $H^4_{\textrm{\rm\'et}}(X\times_k\overline k,\mathbb Z_2(2))^{\mathrm{Gal}(\overline k/k)}$ which are not algebraic non-torsion, too? Previous counterexamples to the integral Tate conjecture (see \cite{CTSz} and \cite{T1}) were all torsion. However we will show that
\begin{thm} The cohomology group $H^3_{\textrm{\rm\'et}}(X\times_k\overline k,\mathbb Z_2(2))$ is trivial and the cohomology group $H^4_{\textrm{\rm\'et}}(X\times_k\overline k,\mathbb Z_2(2))$ has no torsion (where we use the same notations as in Theorem 1.2).
\end{thm}
As the Hochshild--Serre spectral sequence
$$H^i(\mathrm{Gal}(\overline k/k),H^j_{\textrm{\rm\'et}}(X\times_k\overline k,\mathbb Z_2(2)))\Rightarrow H^{i+j}_{\textrm{\rm\'et}}(X,\mathbb Z_2(2))$$
furnishes a short exact sequence:
$$H^1(\mathrm{Gal}(\overline k/k),H^3_{\textrm{\rm\'et}}(X\times_k\overline k,\mathbb Z_2(2)))\rightarrow
H^4_{\textrm{\rm\'et}}(X,\mathbb Z_2(2))\rightarrow H^4_{\textrm{\rm\'et}}(X\times_k\overline k,\mathbb Z_2(2))^{\mathrm{Gal}(\overline k/k)},$$
we immediately get the following
\begin{cor} The cycle class map:
\begin{equation}
CH^2(X)\otimes\mathbb Z_2\rightarrow H^4_{\textrm{\rm\'et}}(X\times_k\overline k,\mathbb Z_2(2))^{\mathrm{Gal}(\overline k/k)}
\end{equation}
is not surjective and hence there are non-algebraic classes in the torsion-free group $H^4_{\textrm{\rm\'et}}(X\times_k\overline k,\mathbb Z_2(2))^{\mathrm{Gal}(\overline k/k)}$.
\end{cor}
The study of such quadratic fibrations has a venerable history (see \cite{AM} and \cite{CTOj}). It is important to note that the properties in Theorem 1.3 are particular to these geometrically rational $5$-folds. In fact in the last section we will show that there are smooth $5$-dimensional projective rational varieties over every algebraically closed field of odd characteristic which have $2$-torsion in its fourth dimensional \'etale cohomology (Proposition 6.7). 
\begin{cont} In the next section we show that there is a choice of $l_1$ and $l_2$ such that $L$ is in general position when the cardinality of $k$ is sufficiently large. In the third section we use unramified cohomology to prove Theorem 1.2, assuming the existence of a regular model. We construct such a model by writing down a quadratic fibration $Y$ over $\mathbb P^2_k$ whose generic fibre is $Q$ and by explicitly computing the blow-up of its singularities in the fourth section. Next we compute the \'etale cohomology of the quadratic threefolds which appear as fibres of $Y$ using rather standard tools. In the last section we prove Theorem 1.3 by a repeated application of the Leray-Serre spectral sequence, exploiting that the exceptional divisors of the blow-ups are quadratic hypersurfaces or families of quadratic hypersurfaces over projective lines.
\end{cont}
\begin{ackn} I wish to thank Alena Pirutka for suggesting to use the results of \cite{KRS1} to show the finiteness of $H^3_{nr}(k(X)/k,\mathbb Z/2\mathbb Z)$ in Theorem 3.3, and for calling my attention to the fact that this result with Theorem 2.2 of \cite{CTK1} implies the second half of Theorem 1.2. I also wish to thank Jean-Louis Colliot-Th\'el\`ene for some useful correspondence about this article. The author was partially supported by the EPSRC grant P36794.
\end{ackn}

\section{Lines in general position}

In all that follows we will denote linear forms and the lines defined by them by the same symbol, by slight abuse of notation. Recall that $H$ is the set of linear forms
$$h=e_xx+e_yy+e_zz$$
where $e_x,e_y,e_z\in\{0,1\}$. 
\begin{lemma} Assume that $k$ has odd characteristics. Then there is a non-empty Zariski-open set $U\subset\mathbb A_k^3\times
\mathbb A_k^3$ such that for every closed point $(l_1,l_2)$ of $U$ the set
$$L=\{x=0\}\cup\{y=0\}\cup\{l_1+h=0|h\in H\}\cup\{l_2+h=0|h\in H\}$$
is a set of lines and it is in general position.
\end{lemma}
\begin{proof} First we are going to prove the following similar claim. Let $R$ be three pair-wise non-propositional linear forms and let $T\subset R$ be a non-empty subset.  Then there is a non-empty Zariski-open set $V\subset\mathbb A_k^3$ such that for every closed point $l$ of $V$ the set
$$R'=\{l+h=0|h\in T\}\cup\{h=0|h\in R-T\}$$
is in general position. Since being in general position is an open condition, there is nothing to prove if $R$ is already in general position. Hence we may assume that the elements of $R$ intersect in one point. After a projective transformation we may assume that $R=\{x=0,y=0,x+y=0\}$ and $T$ is either $\{x=0\},\{x=0,y=0\}$ or $R$ itself. Let $l$ be the linear form $z$; in all three cases $R'$ is in general position with this choice. (Here we use that $k$ has odd characteristics). Since being in general position is an open condition, the claim is now clear. 

Now let us return to the proof of the claim in the lemma. Because being in general position is an open condition, and $\mathbb A_k^3\times
\mathbb A_k^3$ is irreducible, it will be enough to show that for every three element subset $R\subset L$ the set of $(l_1,l_2)\in\overline k^3\times\overline k^3$ such that $R$ is in general position is non-empty. If either $R\cap\{l_1+h=0|h\in H\}$ or $R\cap\{l_2+h=0|h\in H\}$ is empty then this claim follows immediately from the claim above. Otherwise $R$ is of the form $\{l=0,l_1+h=0,l_2+h'=0\}$ for some $h,h'\in H$. There is a choice of $l_1',l_2$ such that the linear forms $l,l_1'+h,l_2+h'$ are pair-wise non-propositional. Fix such an $l_2$; then there is a choice of $l_1=l_1'+l_1''$ such that $R$ is in general position by the claim proved above. 
\end{proof}
Let $U\subset\mathbb A_k^3\times\mathbb A_k^3$ denote the largest non-empty Zariski-open set such that for every closed point $(l_1,l_2)$ of $U$ the set
$$L=\{x=0\}\cup\{y=0\}\cup\{l_1+h=0|h\in H\}\cup\{l_2+h=0|h\in H\}$$
is a set of lines and it is in general position.
\begin{prop} There is a natural number $C$ such that for every field $k$ whose cardinality is at least $C$ the set:
$$U(k)\subseteq\{\big((b_1,c_1,d_1),(b_2,c_2,d_2)\big)\in k^3\times k^3|b_i,c_i,d_i\neq0,-1\textrm{ for $i =1,2$}\}$$
is non-empty.
\end{prop}
\begin{proof} The claim is obvious when $k$ is infinite, so we may assume that $k$ is finite. Note that given three linear forms:
$$f_j=p_jx +q_jy+r_jz,\quad p_j,q_j,r_j\in k,j=1,2,3,$$
the set $\{f_1,f_2,f_3\}$ is in general position if and only if the determinant:
$$\begin{vmatrix}
p_1& q_1 & r_1 \\ p_2 & q_2 & r_2 \\ p_3 & q_3 & r_3
\end{vmatrix}$$
is non-zero. Therefore $U$ is the complement of $\binom{18}{3}=816$ hypersurfaces of degree $3$ in $\mathbb A^6_k$ and $64$ points. So by Lemma 2.3 below
$$|U|\geq|k|^6-64-816\cdot 18|k|^5.$$
The claim is now clear.
\end{proof}
\begin{lemma} Let $H\subset\mathbb A^n_k$ be hypersurface of degree $d$. Then 
$$|H(k)|\leq\min(|k|^n,dn|k|^{n-1}).$$
\end{lemma}
\begin{proof} Without the loss of generality we may assume that $dn<|k|$. We are going to prove the claim by induction on $n$. The case $n=1$ is obvious. Assume now that the claim holds for $n-1$ and let $x$ be a coordinate function on $\mathbb A^n_k$. For every $c\in k$ let $H_c=H\cap\{x=c\}$. Then either $H_c\subset\mathbb A^{n-1}_k$ is a hypersurface of degree at most $d$ and hence
$$|H_c(k)|\leq\min(|k|^{n-1},d(n-1)|k|^{n-2})=d(n-1)|k|^{n-2},$$
by the induction hypothesis or $H_c$ is the hyperplane $\{x=c\}$. This is only possible for at most $d$ values of $c$ and hence
$$|H(k)|=\sum_{c\in k}|H_c(k)|\leq d|k|^{n-1}+|k|d(n-1)|k|^{n-2}=dn|k|^{n-1}.$$
\end{proof}

\section{Unramified cohomology}

\begin{defn} Let $X$ be a smooth, projective variety defined over a field $F$. Let $F(X)$ and $X(d)$ denote the function field of $X$ and the set of points of $X$ of codimension $d$, respectively. For every $x\in X(1)$ let $\mathcal O_{X,x}$ and $F_x$ denote the discrete valuation ring in $F(X)$ corresponding to $x$ and the residue field of $\mathcal O_{X,x}$, respectively. Let $i,j,n$ be natural numbers and assume that $\textrm{char}(F)\!\!\!\not\!|n$. The unramified cohomology group $H_{nr}^i(F(X)/F,\mu_n^{\otimes j})$ of $X$ over $F$ is by definition (see section 1.1 of \cite{CTK1} on page 6):
$$\begin{CD}
H_{nr}^i(F(X)/F,\mu_n^{\otimes j})=\textrm{Ker}
\big(H^i(F(X),\mu_n^{\otimes j})\!@>\oplus_{x\in X(1)}\partial_x>>
\!\bigoplus_{x\in X(1)}H^{i-1}(F_x,\mu_n^{\otimes(j-1)})\big)
\end{CD}$$
where
$$\partial_x:H^i(F(X),\mu_n^{\otimes j})\longrightarrow
H^{i-1}(F_x,\mu_n^{\otimes(j-1)})$$
is the residue map associated to the discrete valuation ring $\mathcal O_{X,x}$. Note that the image of the restriction map
$$H^i(F,\mu_n^{\otimes j})\longrightarrow 
H^i(F(X),\mu_n^{\otimes j})$$
of Galois cohomology actually lies in $H_{nr}^i(F(X)/F,\mu_n^{\otimes j})$. Let
$$\eta^{i,j}_n:H^i(F,\mu_n^{\otimes j})\longrightarrow 
H_{nr}^i(F(X)/F,\mu_n^{\otimes j})$$
be the corresponding map. When $n=2$ we will drop the unnecessary index $j$.
\end{defn}
Assume now that $F$ has odd characteristic and for every non-degenerate quadratic form $\phi$ of dimension $m$ defined over $F$ let $X_\phi$ denote the projective smooth quadric in $\mathbb P^{m-1}_F$ defined by $\phi$. 
\begin{thm} Let $\phi$ be a non-degenerate quadratic form:
$$x_0^2+a_1x^2_1+a_2x^2_2+a_1a_2x^2_3+a_3x^2_4$$
of dimension $5$, where $a_1,a_2,a_3\in F^*$. Then both the kernel and the cokernel of the map:
\begin{equation}
\eta^3_2:H^3(F,\mathbb Z/2\mathbb Z)\longrightarrow 
H^3_{nr}(F(X_{\phi})/F,\mathbb Z/2\mathbb Z)
\end{equation}
are finite.
\end{thm}
\begin{proof} Note that $\phi$ is the neighbour of the Pfister $3$-form $\langle1,a_1\rangle\otimes\langle1,a_2\rangle\otimes\langle1,a_3\rangle$. Therefore the order of Ker$(\eta^3_2)$ is at most two by a theorem of Arason (see Th\'eor\`eme 1.8 of \cite{CTOj} on page 147). Moreover by Proposition 3 (page 845) and Theorem 5 (page 846) of \cite{KRS1} the group Coker$(\eta^3_2)$ is isomorphic to a subgroup of Ker$(\eta^3_2)$, and hence it is also finite.
\end{proof}
Recall that $Q$ is the smooth, projective quadric over the function field of $\mathbb P^2_k$ defined by the homogeneous equation (1.0.1).
\begin{thm} Assume that $k$ is a finite field. Also suppose that $b_i,c_i,d_i\neq0,-1$ for $i =1,2$, the set of lines $L$ is in general position, and let $X$ be a projective smooth variety over $k$ which has a fibration over $\mathbb P^2_k$ whose generic fibre is the quadric $Q$. Then the cohomology group $H^3_{nr}(k(X)/k,\mathbb Z/2\mathbb Z)$ is finite and non-zero.
\end{thm}
Jean-Louis Colliot-Th\'el\`ene informed me that he can easily deduce the first half of this theorem in bigger generality from the results of \cite{CTK1}. However I prefer to give an independent proof since it can be refined to give an easily computable upper bound. 
\begin{proof} Let $F$ be the function field of $\mathbb P^2_k$. Consider the image $\xi$ of the cup product:
$$\delta(a)\cup\delta(f)\cup\delta(g_1)\in H^3(F,\mathbb Z/2\mathbb Z)$$
with respect to $\eta^3_2$ in $H^3(F(Q),\mathbb Z/2\mathbb Z)=H^3(k(X),\mathbb Z/2\mathbb Z)$ where for every $e\in F^*$ we let $\delta(e)$ denote its class in $H^1(F,\mathbb Z/2\mathbb Z)=F^*/(F^*)^2$. The argument of Proposition 3.6 of \cite{Pi1} shows without any modification that $\xi$ is non-zero and lies in $H^3_{nr}(k(X)/k,\mathbb Z/2\mathbb Z)$. (Note that the conditions on page 814 of \cite{Pi1} are satisfied, as we assumed that $L$ is in general position.) Therefore the latter is non-zero. So we only have to show that $H^3_{nr}(k(X)/k,\mathbb Z/2\mathbb Z)$ is finite. By Theorem 3.2 it will be enough to show that the pre-image $M\subset H^3(F,\mathbb Z/2\mathbb Z)$ of the subgroup
$$H^3_{nr}(k(X)/k,\mathbb Z/2\mathbb Z)\subset H^3_{nr}(F(Q)/F,\mathbb Z/2\mathbb Z)$$
with respect to the map $\eta^3_2$ in (3.2.1) is finite. We will need the following\noqed
\end{proof}
\begin{lemma} Let $N\subset\! H^3(F,\mathbb Z/2\mathbb Z)$ be the subgroup of all elements $\zeta\in\! H^3(F,\mathbb Z/2\mathbb Z)$ such that $\partial_x(\zeta)=0$ for every $x\in\mathbb P^2_k(1)$ which does not lie on the discriminant locus of $Q$. Then $M\subseteq N$.
\end{lemma}
\begin{proof} Let $\zeta\in M$ and let $x$ be as above. Let $y\in X(1)$ be the codimension $1$ point of $X$ lying above $x$. There is a commutative diagram:
$$\begin{CD}
H^3(F,\mathbb Z/2\mathbb Z)@>\partial_x>> H^2(k_x,\mathbb Z/2\mathbb Z)\\
@V\eta^3_2VV@VV\eta^2_2V\\
H^3(F(Q),\mathbb Z/2\mathbb Z)@>\partial_y>> H^2(k_y,\mathbb Z/2\mathbb Z).\\
\end{CD}$$
(See page 143 of \cite{CTOj} for an explanation.) Therefore $\eta^2_2\circ\partial_x(\zeta)=0$. But the reduction of $Q$ onto $x$ is a non-degenerate quadric by assumption, so the homomorphism $\eta^2_2:H^2(k_x,\mathbb Z/2\mathbb Z)\rightarrow H^2(k_y,\mathbb Z/2\mathbb Z)$ is injective by part (4) of Corollary 2 of \cite{KRS1} on page 844. 
\end{proof}
\begin{proof}[End of the proof of Theorem 3.3] By Lemma 3.4 it will be sufficient to prove that $N$ is finite. Let $Y$ be a smooth, geometrically integral projective surface over the finite field $k$. Then there is an exact sequence:
$$\begin{CD}
0@>>>H^3(k(Y),\mathbb Z/2\mathbb Z)
@>\bigoplus_{x\in Y(1)}\partial_x>>
\bigoplus_{x\in Y(1)}H^2(k_x,\mathbb Z/2\mathbb Z)\end{CD}$$
$$\begin{CD}@>\sum_{x\in Y(1)}\bigoplus_{y\in\overline x\cap Y(2)}\partial_y>>
\bigoplus_{y\in Y(2)}H^1(k_y,\mathbb Z/2\mathbb Z)
@>>>\mathbb Z/2\mathbb Z\rightarrow 0.\end{CD}$$
by Theorem 0.4 of \cite{KeSa} on page 3 (see also \cite{CTSaSo}). (This is a special case of a more general result, formally known as Kato's conjecture.) For every smooth, geometrically integral projective curve $C$ and for every finite set $S\subset C(1)$ let
$$H^2_S(k(C),\mathbb Z/2\mathbb Z)=\bigcap_{y\in C(1)-S}\!\!\!\!\!\!
\textrm{Ker}(\partial_y)\ \ \subseteq H^2(k(C),\mathbb Z/2\mathbb Z).$$
The group $H^2_S(k(C),\mathbb Z/2\mathbb Z)$ is finite by the local-global principle for quaternion algebras over global fields. By the above there is an exact sequence:
$$\begin{CD}
0\rightarrow N
@>\bigoplus_{l\in L}\partial_l>>
\bigoplus_{l\in L}H^2_{l\cap I}(k(l),\mathbb Z/2\mathbb Z)@>\sum_{l\in L}\bigoplus_{y\in l\cap I}\partial_y>>
\bigoplus_{y\in I}H^1(k_y,\mathbb Z/2\mathbb Z),\end{CD}$$
where $I$ is the set of intersection points of any two elements of $L$. Because the set $L$ is finite, the claim is now clear.
\end{proof}
\begin{proof}[Proof of Theorem 1.2] Let $Y$ be a smooth, projective variety defined over a field $F$. Let $i,j$ be natural numbers and let $l\neq\textrm{char}(F)$ be a prime number. Let
$H_{nr}^i(F(Y)/F,\mathbb Q_l/\mathbb Z_l(j))$ denote the direct limit:
$$H_{nr}^i(F(Y)/F,\mathbb Q_l/\mathbb Z_l(j))=\varinjlim_{n=1}^{\infty}H_{nr}^i(F(Y)/F,\mu_{l^n}^{\otimes j}).$$
Since over the quadratic field extension $k'=k(\sqrt a)$ the variety $X$ becomes rational, we have $H_{nr}^3(k'(X)/k',\mathbb Q_2/\mathbb Z_2(2))=0$. Therefore $H_{nr}^3(k(X)/k,\mathbb Q_2/\mathbb Z_2(2))$ lies in the kernel of the restriction map:
$$H^3(k(X),\mathbb Q_2/\mathbb Z_2(2))\longrightarrow
H^3(k'(X),\mathbb Q_2/\mathbb Z_2(2)),$$
so the group $H_{nr}^3(k(X)/k,\mathbb Q_2/\mathbb Z_2(2))$ is two-torsion. Also note that
$$H_{nr}^3(k(X)/k,\mathbb Q_2/\mathbb Z_2(2))[2]=
H^3_{nr}(k(X)/k,\mathbb Z/2\mathbb Z)$$
by the Merkurjev-Suslin theorem (see Remark 5.6 of \cite{CTK1} on page 29). So we conclude that
$$H_{nr}^3(k(X)/k,\mathbb Q_2/\mathbb Z_2(2))=
H_{nr}^3(k(X)/k,\mathbb Q_2/\mathbb Z_2(2))[2]=
H^3_{nr}(k(X)/k,\mathbb Z/2\mathbb Z).$$
By Theorem 2.1 of \cite{Pi1} on page 806 there is an isomorphism:
$$H^3_{nr}(k(X)/k,\mathbb Q_2/\mathbb Z_2(2))[2]\cong
\mathrm{Coker}[CH^2(X)\rightarrow CH^2(X\times_k\overline k)^{\mathrm{Gal}(\overline k/k)}][2],$$
and hence the map in (1.2.1) is not surjective by Theorem 3.3. By Theorem 2.2 of \cite{CTK1} on page 10 we know that
$$H^3_{nr}(k(X)/k,\mathbb Q_2/\mathbb Z_2(2))^{codiv}\cong\mathrm{Coker}[CH^2(X)\otimes\mathbb Z_2\rightarrow H^4_{\textrm{\rm\'et}}(X,\mathbb Z_2(2))]_{tor}$$
where for every abelian group $M$ we let $M^{codiv}$ and $M_{tor}$ denote the quotient of $M$ modulo its maximal divisible subgroup and the maximal torsion subgroup of $M$, respectively. Therefore
$$H^3_{nr}(k(X)/k,\mathbb Z/2\mathbb Z)^{codiv}\cong\mathrm{Coker}[CH^2(X)\otimes\mathbb Z_2\rightarrow H^4_{\textrm{\rm\'et}}(X,\mathbb Z_2(2))]_{tor}.$$
Since $H^3_{nr}(k(X)/k,\mathbb Z/2\mathbb Z)$ is finite by Theorem 3.3 we get that its maximal divisible subgroup is trivial, and hence map in (1.2.2) is not surjective either. 
\end{proof}

\section{Construction of the regular model}

Let $U_x,U_y,U_z\subset\mathbb P^2_k$ be the open subschemes:
$$U_x=\{x\neq0\},\quad U_y=\{y\neq0\},\quad U_z=\{z\neq0\}.$$
Let $Q_x,Q_y,Q_z$ be the quadrics on $U_x,U_y,U_z$, given by the equations:
$$Q_x:\quad x^2_0-ax^2_1-f_xx^2_2+af_xx^2_3-g_{1x}g_{2x}x^2_4=0,$$

$$Q_y:\quad x^2_0-ax^2_1-f_yx^2_2+af_yx^2_3-g_{1y}g_{2y}x^2_4=0,$$

$$Q_z:\quad x^2_0-ax^2_1-f_zx^2_2+af_zx^2_3-g_{1z}g_{2z}x^2_4=0,$$
respectively, where
$$f_x=\frac{y}{x},\quad g_{1x}=\frac{\prod_{h\in H}(l_1+h)}{x^8},\quad g_{2x}=\frac{\prod_{h\in H}(l_2+h)}{x^8},$$
$$f_y=\frac{x}{y},\quad g_{1y}=\frac{\prod_{h\in H}(l_1+h)}{y^8},\quad g_{2y}=\frac{\prod_{h\in H}(l_2+h)}{y^8},$$
$$f_z=\frac{xy}{z^2},\quad g_{1z}=\frac{\prod_{h\in H}(l_1+h)}{z^8},\quad g_{2z}=\frac{\prod_{h\in H}(l_2+h)}{z^8}.$$
Clearly $Q_x,Q_y,Q_z$ are all isomorphic to $Q$ over $U_x\cap U_y\cap U_z$. There are isomorphisms $Q_x\rightarrow Q_y,Q_y\rightarrow Q_z,Q_z\rightarrow Q_x$ over $U_x\cap U_y,U_y\cap U_z,U_z\cap U_x$, respectively, induced by the change of coordinates:
\begin{equation}
[x_0:x_1:x_2:x_3:x_4]\mapsto[x_0:x_1:\frac{y}{x}x_2:\frac{y}{x}x_3:\frac{y^8}{x^8}x_4],
\end{equation}
\begin{equation}
[x_0:x_1:x_2:x_3:x_4]\mapsto[x_0:x_1:\frac{z}{y}x_2:\frac{z}{y}x_3:\frac{z^8}{y^8}x_4],
\end{equation}
\begin{equation}
[x_0:x_1:x_2:x_3:x_4]\mapsto[x_0:x_1:\frac{x}{z}x_2:\frac{x}{z}x_3:\frac{x^8}{z^8}x_4],
\end{equation}
respectively. These maps satisfy the cocycle condition, and hence the maps 
$Q_x\rightarrow U_x,Q_y\rightarrow U_y,Q_z\rightarrow U_z$ glue together to a fibration $\rho:Y\rightarrow\mathbb P^2_k$ in quadrics.
\begin{lemma} The open subscheme $Y_{xy}=\rho^{-1}(U_x\cap U_y)\subset Y$ is regular, except over the fibres of the set $D$ of intersections of any two of the lines:
$$K=\{l_1+h=0|h\in H\}\cup\{l_2+h=0|h\in H\}.$$
For every $d\in D$ the fibre $\rho^{-1}(d)$ has one isolated singularity which can be resolved with one blow-up. 
\end{lemma}
\begin{proof} We may work over $\overline k$ in all that follows. For every $p\in U_x\cap U_y-\bigcup_{l\in K}\{l=0\}$ the family $Y_{xy}$ is given by the equation:
$$x^2_0+x^2_1+x^2_2+x^2_3+x^2_4=0$$
in an \'etale neighbourhood around $p$, after a suitable change of coordinates. In particular over this open subscheme the map $\rho$ is smooth. For every $p\in\bigcup_{l\in K}\{l=0\}-D$ the family $Y_{xy}$ is given by the equation:
$$x^2_0+x^2_1+x^2_2+x^2_3+tx^2_4=0$$
in an \'etale neighbourhood around $p$ where $t$ is a local coordinate around $p$, after a suitable change of coordinates. By taking partial derivatives we get that $\rho^{-1}(U_x\cap U_y-D)$ is still regular and over $\bigcup_{l\in K}\{l=0\}-D$ the fibre of $\rho$ is a cone over a non-singular quadric surface. For every $p\in D$ the family $Y_{xy}$ is given by the equation:
$$x^2_0+x^2_1+x^2_2+x^2_3+t_1t_2x^2_4=0$$
in an \'etale neighbourhood around $p$ where $t_1,t_2$ are a complete system of local coordinates around $p$. By taking partial derivatives we get that there is an isolated singularity at $x_0=\cdots=x_3=t_1=t_2=0$. Consider the affine open neighbourhood of this singularity given by the equation:
$$x^2_0+x^2_1+x^2_2+x^2_3+t_1t_2=0,$$
and consider the blow-up of this variety at $(0,0,\ldots,0)$ in $\mathbb A^6_{\overline k}\times\mathbb P^5_{\overline k}$, where we introduce the homogeneous coordinates $[y_0:\ldots:y_3:u_1:u_2]$ for $\mathbb P^5_{\overline k}$. On the open subscheme $y_0\neq0$ the equation of the blow-up is:
$$1+y^2_1+y^2_2+y^2_3+u_1u_2=0$$
in the variables $x_0,y_1,y_2,y_3,u_1,u_2$, and we have a similar equation for the domains $y_1\neq0,y_2\neq0,y_3\neq0$, too. On the open subscheme $u_1\neq0$ the equation of the blow-up is:
$$y^2_0+y^2_1+y^2_2+y^2_3+u_2=0$$
in the variables $t_1,y_0,\ldots,y_3,u_2$, and we have a similar equation for the domain $u_2\neq0$, too. By taking partial derivatives we get that these varieties are smooth.
\end{proof}
Let $V_{xy}$ denote the closed subscheme of $\mathbb P^2_k$ which is the complement of $U_{x}\cap U_{y}$, $\textrm{i.e.}$ given by the equation $xy=0$. Let $Z_x,Z_y,Z_z$ denote the closed subscheme of $Q_x,Q_y,Q_z,$ respectively, given by the same system of equations:
$$xy=0,\quad x^2_2=ax^2_3,\quad x_0=x_1=x_4=0.$$
These glue together under the maps (4.0.1), (4.0.2) and (4.0.3), hence they define a closed subscheme $Z\subseteq Y$ which lies in $\rho^{-1}(V_{xy})$.
\begin{lemma} The restriction of the singular locus of $Y$ onto $\rho^{-1}(V_{xy}-[0:0:1])$ is the closed subscheme $Z$. By blowing up $Z$ the singularities are resolved.
\end{lemma}
\begin{proof} For every $p\in V_{xy}-\{[0:0:1]\}-\bigcup_{l\in K}l$ the family $Y$ is given by the equation:
$$r_0r_1t_1+x^2_0+x^2_1+x^2_4=0$$
in an \'etale neighbourhood around $p$ where $t_1,t_2$ are a complete system of local coordinates around $p$ such that the closed subscheme $Z$ is given by the equations:
$$t_1=0,\quad r_0r_1=0,\quad x_0=x_1=x_4=0.$$
By taking partial derivatives we get that the singular locus is indeed $Z$ which is the union of the two closed subsets:
$$Z_0:r_0=t_1=x_0=x_1=x_4=0,$$
$$Z_1:r_1=t_1=x_0=x_1=x_4=0.$$
Consider the affine open neighbourhood $r_1\neq0$ of $Z_0$ given by the equation:
$$r_0t_1+x^2_0+x^2_1+x^2_4=0,$$
and consider the blow-up of this variety at $Z_0$ in
$\mathbb A^6_{\overline k}\times\mathbb P^4_{\overline k}$, where we introduce the homogeneous coordinates $[s_0:v_1:y_0:y_1:y_4]$ for $\mathbb P^4_{\overline k}$ (corresponding to $(r_0,t_1,x_0,x_1,x_4)$). On the open subscheme $v_1\neq0$ the equation of the blow-up is:
$$s_0+y^2_0+y^2_1+y^2_4=0$$
in the variables $t_1,t_2,s_0,y_0,y_1,y_4$, and we have a similar equation for the domain $s_0\neq0$. By taking partial derivatives we get that these varieties are smooth. On the open subscheme $y_0\neq0$ the equation of the blow-up is:
$$s_0v_1+1+y^2_1+y^2_4=0$$
in the variables $v_1,t_2,s_0,x_0,y_1,y_4$, and we have a similar equation for the domains $y_1\neq0$ and $y_4\neq0$. These are smooth, too. A similar computation takes care of $Z_1$.

For every $p\in V_{xy}\cap\left(\bigcup_{l\in K}l\right)$ the family $Y$ is given by the equation:
$$r_0r_1t_1+x^2_0+x^2_1+x^2_4t_2=0$$
in an \'etale neighbourhood around $p$ where $t_1,t_2$ are a complete system of local coordinates around $p$ such that the closed subscheme $Z$ is given by the equations:
$$t_1=0,\quad r_0r_1=0,\quad x_0=x_1=x_4=0.$$
By taking partial derivatives we get that the singular locus is still $Z$ which is the union of the two closed subsets:
$$Z_0:r_0=t_1=x_0=x_1=x_4=0,$$
$$Z_1:r_1=t_1=x_0=x_1=x_4=0.$$
Consider the affine open neighbourhood $r_1\neq0$ of $Z_0$ given by the equation:
$$r_0t_1+x^2_0+x^2_1+x^2_4t_2=0,$$
and consider the blow-up of this variety at $Z_0$ in
$\mathbb A^6_{\overline k}\times\mathbb P^4_{\overline k}$, where we introduce the homogeneous coordinates $[s_0:v_1:y_0:y_1:y_4]$ for $\mathbb P^4_{\overline k}$ (corresponding to $(r_0,t_1,x_0,x_1,x_4)$). On the open subscheme $v_1\neq0$ the equation of the blow-up is:
$$s_0+y^2_0+y^2_1+y^2_4t_2=0$$
in the variables $t_1,t_2,s_0,y_0,y_1,y_4$, and we have a similar equation for the domain $s_0\neq0$. By taking partial derivatives we get that these varieties are smooth. On the open subscheme $y_0\neq0$ the equation of the blow-up is:
$$s_0v_1+1+y^2_1+y^2_4t_2=0$$
in the variables $v_1,t_2,s_0,x_0,y_1,y_4$, and we have a similar equation for the domains $y_1\neq0$ and $y_4\neq0$. These are smooth, too. A similar computation takes care of $Z_1$.
\end{proof}
Let $W\subset\rho^{-1}([0:0:1])$ be the closed subscheme which is given by the system of equations:
$$x=y=x_0=x_1=x_4=0,$$
if we consider it as a subscheme of $Q_z$.
\begin{lemma} The restriction of the singular locus of $Y$ onto $\rho^{-1}([0:0:1])$ is the closed subscheme $W$. By blowing up $W$ first, then the proper transform
of $Z$, these singularities are resolved.  
\end{lemma}
\begin{proof} At $p=[0:0:1]$ the family $Y$ is given by the equation:
$$r_0r_1t_1t_2+x^2_0+x^2_1+x^2_4=0$$
in an \'etale neighbourhood around $p$ where $t_1,t_2$ are a complete system of local coordinates around $p$ such that the closed subscheme $Z$ is given by the equations:
$$t_1t_2=0,\quad r_0r_1=0,\quad x_0=x_1=x_4=0,$$
while $W$ is given by the equations:
$$t_1=t_2=x_0=x_1=x_4=0.$$
By taking partial derivatives we get that the singular locus is the union of $Z$ and $W$. On the affine open $r_0\neq 0$ the family $Y$ is given by the equation:
$$r_1t_1t_2+x^2_0+x^2_1+x^2_4=0.$$
Consider the blow-up of this variety at $W$ in
$\mathbb A^6_{\overline k}\times\mathbb P^4_{\overline k}$, where we introduce the homogeneous coordinates $[v_1:v_2:y_0:y_1:y_4]$ for $\mathbb P^4_{\overline k}$  (corresponding to $(t_1,t_2,x_0,x_1,x_4)$). On the open subscheme $v_1\neq0$ the equation of the blow-up is:
\begin{equation}
r_1v_2+y^2_0+y^2_1+y^2_4=0
\end{equation}
in the variables $r_1,t_1,v_2,y_0,y_1,y_4$, and we have a similar equation for the domain $v_2\neq0$. By taking partial derivatives we get that these varieties are smooth except at the closed subvarieties $Z'_1$ and $Z'_2$ given by the system of equations
$$Z'_1:r_1=v_2=y_0=y_1=y_4=0\quad\textrm{and}\quad Z'_2: r_1=v_1=y_0=y_1=y_4=0,$$
respectively. On the open subscheme $y_0\neq0$ the equation of the blow-up is:
$$r_1v_1v_2+1+y^2_1+y^2_4=0$$
in the variables $r_1,v_1,v_2,x_0,y_1,y_4$, and we have a similar equation for the domains $y_1\neq0$ and $y_4\neq0$. These are smooth. 

Consider the blow-up of the variety given the equation in (4.3.1) at $Z'_1$ in
$\mathbb A^6_{\overline k}\times\mathbb P^4_{\overline k}$, where we introduce the homogeneous coordinates $[s_1:w_2:z_0:z_1:z_4]$ for $\mathbb P^4_{\overline k}$ (corresponding to $(r_1,v_2,y_0,y_1,y_4)$). On the open subscheme $s_1\neq0$ the equation of the blow-up is:
$$w_2+z^2_0+z^2_1+z^2_4=0$$
in the variables $r_1,t_1,w_2,z_0,z_1,z_4$, and we have a similar equation for the domain $w_2\neq0$. By taking partial derivatives we get that these varieties are smooth. On the open subscheme $z_0\neq0$ the equation of the blow-up is:
$$s_1w_2+1+z^2_1+z^2_4=0$$
in the variables $s_1,t_1,w_2,y_0,z_1,z_4$, and we have a similar equation for the domains $y_1\neq0$ and $y_4\neq0$. These are all smooth. A similar computation shows that the blow-up at $Z'_2$ of the variety which we get by blowing-up $Y$ at $W$ is smooth in an open neighbourhood of the pre-image of $Z'_2$. 

It is also clear from the above that the union $Z'_1\cup Z'_2$ is the intersection of the proper transform of the closed subscheme
\begin{equation}
t_1t_2=0,\quad r_1=0,\quad x_0=x_1=x_4=0
\end{equation}
with an open neighbourhood of fibre of the blow-up over $t_1=t_2=0$. A similar computation shows that the singular locus of the blow-up of $Y$ at $W$ on the affine open $r_1\neq 0$ is the intersection of the proper transform of the closed subscheme:
\begin{equation}
t_1t_2=0,\quad r_0=0,\quad x_0=x_1=x_4=0
\end{equation}
with an open neighbourhood of fibre of the blow-up over $t_1=t_2=0$, and the blow-up of this subvariety is regular. Since the (scheme-theoretical)
union of the two closed subschemes in (4.3.2) and (4.3.3) is $Z$, the claim is now clear.
\end{proof}
\begin{defn} Let $\rho_0:Y_0\rightarrow\mathbb P^2_k$ be the family
which we get by blowing up the disjoint closed subschemes $W\subset Y$ and the isolated singularities of $Y_{xy}\subset Y$ inside the family $\rho:Y\rightarrow\mathbb P^2_k$. Let $\pi:X\rightarrow\mathbb P^2_k$
be  be the family which we get by blowing up the proper transform $\widetilde
Z\subset Y_0$ of the closed subscheme $Z\subset Y$ inside the family $\rho_0:Y_0\rightarrow\mathbb P^2_k$. By the three lemmas above the resulting variety $X$ is regular.
\end{defn}
\begin{proof}[Proof of Theorem 1.1] By construction $X$ is a model of $Q$, therefore the claim is true.
\end{proof} 

\section{The cohomology of quadrics}

In this section and the next we will drop the subscript ${\textrm{\rm\'et}}$ when we will talk about \'etale cohomology, for the sake of simple notation. Since we will not talk about any other cohomology theory, this will not lead to confusion.  Let $l$ be a prime different from the characteristic of $k$. Let $C_1\subset\mathbb P^4_{\overline k}$ be a non-singular quadric hypersurface. After a change of variables $C_1$ is defined by the homogeneous equation:
$$x^2_0+x^2_1+x^2_2+x^2_3+x^2_4=0.$$
Let $\eta\in H^2(C_1,\mathbb Z_l)(-1)$ be the class of the hyperplane section.
\begin{prop} The following holds:
\begin{enumerate}
\item[$(a)$] we have:
$$H^i(C_1,\mathbb Z_l)\cong\left\{\begin{array}{ll}\mathbb Z_l,&\text{if $i$ is even and $0\leq i\leq 6$,}
\\0,&\text{otherwise,}\end{array}\right.$$
\item[$(b)$] the class $\eta$ generates $H^2(C_1,\mathbb Z_l)(-1)$ as a $\mathbb Z_l$-module.
\end{enumerate} 
\end{prop}
\begin{proof} See Theorem 1.12 of \cite{Re1} and the discussion preceding it.
\end{proof}
Let $C_2\subset\mathbb P^4_{\overline k}$ be a quadric hypersurface given by the homogeneous equation:
$$x_0^2+x_1^2+x^2_2+x^2_3=0,$$
where $[x_0:x_1:x_2:x_3:x_4]$ are the homogeneous coordinates on $\mathbb P^4_{\overline k}$. Then $C_2$ is a cone over the non-sigular quadric surface $C\subset\mathbb P^3_{\overline k}$ given by the same equation. Let $\widetilde C_2$ be the blow-up of $C_2$ at its unique singular point $[0:0:0:0:1]$.
\begin{prop} The following holds:
$$H^i(\widetilde C_2,\mathbb Z_l)\cong\left\{\begin{array}{ll}
\mathbb Z_l,&\text{if $i=0,6$,}\\
\mathbb Z_l^3,&\text{if $i=2,4$,}\\
0,&\text{otherwise,}\end{array}\right.$$
\end{prop}
\begin{proof} The scheme $\widetilde C_2$ is a $\mathbb P^1$-bundle over $C$. Let $p:\widetilde C_2\rightarrow C$ be this fibration. By the proper base change theorem $p_*\mathbb Z_l\cong\mathbb Z_l$, and we have $R^ip_*\mathbb Z_l=0$ for every $i\neq0,2$. Moreover $R^2p_*\mathbb Z_l\cong\mathbb Z_l$, again by the proper base change theorem and since the fibration has a section, as the Brauer group of $C$ is trivial. Therefore the Leray spectral sequence $H^i(C,R^jp_*\mathbb Z_l)\Rightarrow H^{i+j}(\widetilde C_2,\mathbb Z_l)$ degenerates and the claim follows from the well-known fact that
$$H^i(C,\mathbb Z_l)\cong\left\{\begin{array}{ll}
\mathbb Z_l,&\text{if $i=0,4$,}\\
\mathbb Z_l^2,&\text{if $i=2$,}\\
0,&\text{otherwise,}\end{array}\right.$$
(for a reference, see Theorem 1.12 of \cite{Re1} and the discussion preceding it).
\end{proof}
Let $\eta\in H^2(C_2,\mathbb Z_l)(-1)$ be the class of the hyperplane section and let $b:\widetilde C_2\rightarrow C_2$ denote the blow-up map.
\begin{prop} The following holds:
\begin{enumerate}
\item[$(a)$] we have:
$$H^i(C_2,\mathbb Z_l)=\left\{\begin{array}{ll}
\mathbb Z_l,&\text{if $i=0,2,6$,}\\
\mathbb Z_l^2,&\text{if $i=4$,}\\
0,&\text{otherwise,}\end{array}\right.$$
\item[$(b)$] the class $\eta$ generates $H^2(C_2,\mathbb Z_l)(-1)$ as a $\mathbb Z_l$-module.
\end{enumerate} 
\end{prop}
\begin{proof} Let $i:C\rightarrow \widetilde C_2$ be the section at infinity of the $\mathbb P^1$-bundle $\widetilde C_2\rightarrow C$. Then the image of $i$ is the exceptional divisor of the blow-up $b:\widetilde C_2\rightarrow C_2$. Let $U\subset \widetilde C_2$ be the complement of $i(C)$ and let $j:U\rightarrow \widetilde C_2$ be the inclusion map. Because the functor $i_*$ is exact, we have $H^*(C,\mathbb Z_l)=H^*(\widetilde C_2,i_*\mathbb Z_l)$. Therefore the long cohomological exact sequence associated to the short exact sequence:
$$\begin{CD} 0@>>>j_!\mathbb Z_l@>>>
\mathbb Z_l@>>>i_*\mathbb Z_l@>>>0\end{CD}$$
looks as follows:
$$\begin{CD} 0@>>>\mathbb Z_l@>\alpha>>
\mathbb Z_l@>>>H^1(\widetilde C_2,j_!\mathbb Z_l)@>>>
0@>>>0
\end{CD}$$
$$\begin{CD}
@>>>H^2(\widetilde C_2,j_!\mathbb Z_l)@>>>
\mathbb Z_l^3
@>\beta>>\mathbb Z_l^2@>>>
H^3(\widetilde C_2,j_!\mathbb Z_l)@>>>0
\end{CD}$$
$$\begin{CD}
@>>>0
@>>>H^4(\widetilde C_2,j_!\mathbb Z_l)@>>>\mathbb Z_l^3
@>\gamma>>\mathbb Z_l
@>>>H^5(\widetilde C_2,j_!\mathbb Z_l)
\end{CD}$$
$$\begin{CD}
@>>>0
@>>>0
@>>>
H^6(\widetilde C_2,j_!\mathbb Z_l)@>>>\mathbb Z_l@>>>0.\end{CD}$$
The map $\alpha$ is an isomorphism. The composition of $i:C\rightarrow\widetilde C_2$ and the projection $p:\widetilde C_2\rightarrow C$ is the identity map of $C$, and hence the maps $\beta,\gamma$ are surjective. Therefore
$$H^i(\widetilde C_2,j_!\mathbb Z_l)=\left\{\begin{array}{ll}
\mathbb Z_l,&\text{if $i=2,6$,}\\
\mathbb Z_l^2,&\text{if $i=4$,}\\
0,&\text{otherwise.}\end{array}\right.$$
Let $V\subset C_2$ be the complement of the cusp of the cone $C_2$ and let $j':V\rightarrow C_2$ be the inclusion map. Clearly $b_*j_!\mathbb Z_l=j'_!\mathbb Z_l$ and the higher direct image sheaves $R^ib_*j_!\mathbb Z_l$ vanish for every $i>0$ by the proper base change theorem. Therefore $H^*(C_2,j'_!\mathbb Z_l)=H^*(\widetilde C_2,j_!\mathbb Z_l)$. Since the complement of $V$ is a point, we have:
$$H^i(C_2,\mathbb Z_l)=\left\{\begin{array}{ll}
H^i(C_2,j'_!\mathbb Z_l)\oplus\mathbb Z_l,&\text{if $i=0$,}\\
H^i(C_2,j'_!\mathbb Z_l),&\text{otherwise,}\end{array}\right.$$
Claim $(a)$ is now clear. Because the pull-back of $\eta$ onto any line lying on $C$ is still a generator, claim $(b)$ must hold, too.
\end{proof}
Let $\widetilde C_3$ be the blow-up of the quadric hypersurface $C_3\subset\mathbb P^4_{\overline k}$ given by the homogeneous equation:
$$x_0^2+x_1^2+x_2^2=0$$
at the line $x_0=x_1=x_2=0$. Let $m:\widetilde C_3\rightarrow C_3$ be the blow-up map and let $\eta\in H^2(C_3,\mathbb Z_l)(-1)$ be the class of the hyperplane section.  
\begin{prop} The following holds:
$$H^i(\widetilde C_3,\mathbb Z_l)\cong\left\{\begin{array}{ll}
\mathbb Z_l,&\text{if $i=0,6$,}\\
\mathbb Z_l^2,&\text{if $i=2,4$,}\\
0,&\text{otherwise.}\end{array}\right.$$
\end{prop}
\begin{proof} The variety $\widetilde C_3$ is a $\mathbb P^2$-bundle over $\mathbb P^1_{\overline k}$. Let $p:\widetilde C_3\rightarrow\mathbb P^1_{\overline k}$ be this fibration. By the proper base change theorem $p_*\mathbb Z_l\cong\mathbb Z_l$, and we have $R^ip_*\mathbb Z_l=0$ for every $i\neq0,2,4$. Moreover $R^2p_*\mathbb Z_l\cong R^4p_*\mathbb Z_l\cong\mathbb Z_l$, again by the proper base change theorem and since the associated fibration of hyperplane sections has a section, as the Brauer group of $\mathbb P^1_{\overline k}$ is trivial. Therefore the Leray spectral sequence $H^i(\mathbb P^1_{\overline k},R^jp_*\mathbb Z_l)\Rightarrow H^{i+j}(\widetilde C_3,\mathbb Z_l)$ degenerates and the claim follows. 
\end{proof} 
\begin{prop} The following holds:
\begin{enumerate}
\item[$(a)$] we have:
$$H^i(C_3,\mathbb Z_l)\cong\left\{\begin{array}{ll}
\mathbb Z_l,&\text{if $i=0,2,4,6$,}\\
0,&\text{otherwise,}\end{array}\right.$$
\item[$(b)$] the class $\eta$ generates $H^2(C_3,\mathbb Z_l)(-1)$ as a $\mathbb Z_l$-module.
\end{enumerate} 
\end{prop}
\begin{proof} Let $C'\subset\widetilde C_3$ be the exceptional divisor of the blow-up $m:\widetilde C_3\rightarrow C_3$ and let $p':C\rightarrow\mathbb P^1_{\overline k}$ be the restriction of the fibration $p:\widetilde C_3\rightarrow\mathbb P^1_{\overline k}$ onto $C'$. Then $p'$ equips $C'$ with the structure of $\mathbb P^1$-bundle over $\mathbb P^1_{\overline k}$. It has a section as the Brauer group of $\mathbb P^1_{\overline k}$ is trivial. Therefore by repeating the argument in the proof of Proposition 5.2 we get that the Leray spectral sequence $H^i(\mathbb P^1_{\overline k},R^jp'_*\mathbb Z_l)\Rightarrow H^{i+j}(C',\mathbb Z_l)$ degenerates and
$$H^i(C',\mathbb Z_l)\cong\left\{\begin{array}{ll}
\mathbb Z_l,&\text{if $i=0,4$,}\\
\mathbb Z_l^2,&\text{if $i=2$,}\\
0,&\text{otherwise.}\end{array}\right.$$
Let $i:C'\rightarrow\widetilde C_3$ be the inclusion map; it induces isomorphisms $p_*\mathbb Z_l\rightarrow p'_*\mathbb Z_l$ and $R^2p_*\mathbb Z_l\rightarrow R^2p'_*\mathbb Z_l$ by the proper base change theorem. Therefore the pull-back maps $i^*:H^2(\widetilde C_3,\mathbb Z_l)\rightarrow H^2(C',\mathbb Z_l)$ and $i^*:H^4(\widetilde C_3,\mathbb Z_l)\rightarrow H^4(C',\mathbb Z_l)$ are surjective by the functoriality of the Leray spectral sequence and by the above.

Let $U\subset\widetilde C_3$ be the complement of $C'$ and let $j:U\rightarrow\widetilde C_3$ be the inclusion map. By the proper base change theorem $H^*(C',\mathbb Z_l)=H^*(\widetilde C_3,i_*\mathbb Z_l)$. Therefore the long cohomological exact sequence associated to the short exact sequence:
$$\begin{CD} 0@>>>j_!\mathbb Z_l@>>>
\mathbb Z_l@>>>i_*\mathbb Z_l@>>>0\end{CD}$$
looks as follows:
$$\begin{CD} 0@>>>\mathbb Z_l@>\alpha>>
\mathbb Z_l@>>>H^1(\widetilde C_3,j_!\mathbb Z_l)@>>>
0@>>>0
\end{CD}$$
$$\begin{CD}
@>>>H^2(\widetilde C_3,j_!\mathbb Z_l)@>>>
\mathbb Z_l^2
@>\beta>>\mathbb Z_l^2@>>>
H^3(\widetilde C_3,j_!\mathbb Z_l)@>>>0
\end{CD}$$
$$\begin{CD}
@>>>0
@>>>H^4(\widetilde C_3,j_!\mathbb Z_l)@>>>\mathbb Z_l^2
@>\gamma>>\mathbb Z_l
@>>>H^5(\widetilde C_3,j_!\mathbb Z_l)
\end{CD}$$
$$\begin{CD}
@>>>0
@>>>0
@>>>
H^6(\widetilde C_3,j_!\mathbb Z_l)@>>>\mathbb Z_l@>>>0.\end{CD}$$
The map $\alpha$ is an isomorphism, and the maps $\beta,\gamma$ are surjective, as we saw above. Therefore
$$H^i(\widetilde C_3,j_!\mathbb Z_l)=\left\{\begin{array}{ll}
\mathbb Z_l,&\text{if $i=4,6$,}\\
0,&\text{otherwise.}\end{array}\right.$$
Let $V\subset C_3$ be the complement of the line $x_0=x_1=x_2$ and let $j':V\rightarrow C_3$ be the inclusion map. Clearly $m_*j_!\mathbb Z_l=j'_!\mathbb Z_l$ and the higher direct image sheaves $R^im_*j_!\mathbb Z_l$ vanish for every $i>0$ by the proper base change theorem. Therefore $H^*(C_3,j'_*\mathbb Z_l)=H^*(\widetilde C_3,j_!\mathbb Z_l)$. Let $V'\subset C_3$ be the complement of $V$ and let $i':V'\rightarrow C_3$ be the inclusion map. Then $V'$ is isomorphic to $\mathbb P^1_{\overline k}$. Therefore the long cohomological exact sequence associated to the short exact sequence:
$$\begin{CD} 0@>>>j'_!\mathbb Z_l@>>>
\mathbb Z_l@>>>i'_*\mathbb Z_l@>>>0\end{CD}$$
looks as follows:
$$\begin{CD} 0@>>>H^0(C_3,\mathbb Z_l)@>>>
\mathbb Z_l@>>>0@>>>H^1(C_3,\mathbb Z_l)@>>>0
\end{CD}$$
$$\begin{CD}
@>>>0@>>>
H^2(C_3,\mathbb Z_l)
@>>>\mathbb Z_l@>>>
0@>>>H^3(C_3,\mathbb Z_l)
\end{CD}$$
$$\begin{CD}
@>>>0
@>>>\mathbb Z_l@>>>H^4(C_3,\mathbb Z_l)
@>>>0
@>>>0
\end{CD}$$
$$\begin{CD}
@>>>H^5(C_3,\mathbb Z_l)
@>>>0
@>>>\mathbb Z_l@>>>H^6(C_3,\mathbb Z_l)@>>>0.\end{CD}$$
Claim $(a)$ follows. Because the pull-back of $\eta$ onto any line lying on $C_3$ is still a generator, claim $(b)$ holds, too.
\end{proof}
Let $C_4$ be the quadric hypersurface $\subset\mathbb P^4_{\overline k}$ given by the homogeneous equation:
$$x_0x_1=0$$
$[x_0:x_1:x_2:x_3:x_4]$ are the homogeneous coordinates on $\mathbb P^4_{\overline k}$ and let $\eta\in H^2(C_4,\mathbb Z_l)(-1)$ be the class of the hyperplane section. 
\begin{prop} The following holds:
\begin{enumerate}
\item[$(a)$] we have 
$$H^i(C_4,\mathbb Z_l)\cong\left\{\begin{array}{ll}
\mathbb Z_l,&\text{if $i=0,2,4$,}\\
\mathbb Z_l^2,&\text{if $i=6$,}\\
0,&\text{otherwise.}\end{array}\right.$$
\item[$(b)$] the class $\eta$ generates $H^2(C_4,\mathbb Z_l)(-1)$ as a $\mathbb Z_l$-module.
\end{enumerate} 
\end{prop}
\begin{proof} Let $U=\{x_0=0\}$ and $V=\{x_1=0\}$. Clearly $C_4=U\cup V$. Since $U\cong V\cong\mathbb P^3_{\overline k}$ and $U\cap V\cong\mathbb P^2_{\overline k}$, the Mayer-Vietoris long exact sequence for the triad $(C_4,U,V)$ look as follows:
$$\begin{CD} 0@>>>\mathbb Z_l@>>>
\mathbb Z_l^2@>\alpha>>\mathbb Z_l
@>>>H^1(C_4,\mathbb Z_l)@>>>
0@>>>0
\end{CD}$$
$$\begin{CD}
@>>>H^2(C_4,\mathbb Z_l)@>>>
\mathbb Z_l^2
@>\beta>>\mathbb Z_l@>>>
H^3(C_4,\mathbb Z_l)@>>>0
\end{CD}$$
$$\begin{CD}
@>>>0
@>>>H^4(C_4,\mathbb Z_l)@>>>\mathbb Z_l^2
@>\gamma>>\mathbb Z_l
@>>>H^5(C_4,\mathbb Z_l)
\end{CD}$$
$$\begin{CD}
@>>>0
@>>>0
@>>>
H^6(C_4,\mathbb Z_l)@>>>\mathbb Z_l^2@>>>0.\end{CD}$$
The maps $\alpha,\beta,\gamma$ are surjective and hence claim $(a)$ holds. Because the pull-back of $\eta$ onto any line lying on $C_4$ is still a generator, claim $(b)$ holds, too.
\end{proof}

\section{Cohomology of the regular model}

For every Spec$(k)$-scheme $S$ let $\overline S$ denote its base change to Spec$(\overline k)$. Similarly for every morphism $m:R\rightarrow S$ of
Spec$(k)$-schemes let $\overline m:\overline R\rightarrow\overline S$ denote the base change to Spec$(\overline k)$. Let $\rho$ be the map introduced in Definition 4.4. 
\begin{lemma} The following holds:
\begin{enumerate}
\item[$(a)$] we have
$$R^i\overline{\rho}_*(\mathbb Z_l)\cong\left\{\begin{array}{ll}
\mathbb Z_l,&\text{if $i=0,2$,}\\
0,&\text{if $i=1,3$.}
\end{array}\right.$$
\item[$(b)$] the stalks of $R^4\overline{\rho}_*(\mathbb Z_l)$ are torsion-free.
\end{enumerate} 
\end{lemma}
\begin{proof} Note that the geometric fibres of $\rho$ are all isomorphic to one of the quadrics $C_1,C_2,C_3,$ or $C_4$. Because they are all connected, we have $R^0\overline{\rho}_*(\mathbb Z_l)\cong\mathbb Z_l$ by the proper base change theorem. Moreover by part $(a)$ of Propositions 5.1, 5.3, 5.5 and 5.6 and the proper base change theorem the stalks of the sheaves $R^1\overline{\rho}_*(\mathbb Z_l)$ and $R^3\overline{\rho}_*(\mathbb Z_l)$ are trivial, while the stalks of $R^4\overline{\rho}_*(\mathbb Z_l)$ are torsion-free. It is clear from the construction of the fibration $\rho:Y\rightarrow\mathbb P^2_k$ that there is a $\mathbb P^4$-bundle $\tau:T\rightarrow\mathbb P^2_k$ and a closed immersion of $\mathbb P^2_k$-schemes $\iota:Y\rightarrow T$ such that $\rho=\tau\circ\iota$. The sheaf $R^2\overline{\tau}_*(\mathbb Z_l)$ is lisse of rank one. Because the \'etale fundamental group of $\mathbb P^2_{\overline k}$ is trivial we get that $R^2\overline{\tau}_*(\mathbb Z_l)\cong\mathbb Z_l$. By part $(b)$ of Propositions 5.1, 5.3, 5.5 and 5.6 and the proper base change theorem the map $\overline{\iota}^*:R^2\overline{\tau}_*(\mathbb Z_l)\rightarrow R^2\overline{\rho}_*(\mathbb Z_l)$ is an isomorphism. 
\end{proof}
\begin{prop} The following holds:
\begin{enumerate}
\item[$(a)$] we have
$$H^i(\overline Y,\mathbb Z_l)\cong\left\{\begin{array}{ll}
\mathbb Z_l,&\text{if $i=0$,}\\
0,&\text{if $i=1,3$,}\\
\mathbb Z_l^2,&\text{if $i=2$.}\end{array}\right.$$
\item[$(b)$] the group $H^4(\overline Y,\mathbb Z_l)$ is torsion-free.
\end{enumerate} 
\end{prop}
\begin{proof} There is a Leray-Serre spectral sequence:
$$E^2_{i,j}=H^i(\mathbb P^2_{\overline k},R^j\overline{\rho}_*(\mathbb Z_l))
\Rightarrow
H^{i+j}(\overline Y,\mathbb Z_l).$$
Clearly $H^0(\overline Y,\mathbb Z_l)=\mathbb Z_l$ since $\overline Y$ is connected. By Lemma 6.1 we have 
$$E^2_{1,0}=H^1(\mathbb P^2_{\overline k},\mathbb Z_l)=0\text{ and }E^2_{0,1}=H^0(\mathbb P^2_{\overline k},0)=0,$$
and hence $H^1(\overline Y,\mathbb Z_l)=0$. Similarly by Lemma 6.1 we have
$$E^2_{3,0}=H^3(\mathbb P^2_{\overline k},\mathbb Z_l)=0,\ E^2_{2,1}=H^2(\mathbb P^2_{\overline k},0)=0,$$
$$E^2_{1,2}=H^1(\mathbb P^2_{\overline k},\mathbb Z_l)=0,\text{ and }E^2_{0,3}=H^0(\mathbb P^2_{\overline k},0)=0,$$
and hence $H^3(\overline Y,\mathbb Z_l)=0$. Moreover
$$E^2_{2,0}=H^2(\mathbb P^2_{\overline k},\mathbb Z_l)=\mathbb Z_l,\ E^2_{1,1}=H^1(\mathbb P^2_{\overline k},0)=0,\ E^2_{0,2}=H^0(\mathbb P^2_{\overline k},\mathbb Z_l)=\mathbb Z_l$$
by Lemma 6.1. However the domain and the range, respectively, of the boundary maps $d^2_{0,1}:E^2_{0,1}\rightarrow E^2_{2,0}$ and $d^2_{0,2}:E^2_{0,2}\rightarrow E^2_{2,1}$ are zero, so we get that
$$E^{\infty}_{2,0}\cong E^2_{2,0}\cong\mathbb Z_l\textrm{ and }
E^{\infty}_{0,2}\cong E^2_{0,2}\cong\mathbb Z_l.$$
Therefore we have $H^2(\overline Y,\mathbb Z_l)\cong\mathbb Z_l^2$. By Lemma 6.1
$$E^2_{4,0}=H^4(\mathbb P^2_{\overline k},\mathbb Z_l)=\mathbb Z_l,\ E^2_{3,1}=H^3(\mathbb P^2_{\overline k},0)=0,$$
$$E^2_{2,2}=H^2(\mathbb P^2_{\overline k},\mathbb Z_l)=\mathbb Z_l,\text{ and }E^2_{1,3}=H^1(\mathbb P^2_{\overline k},0)=0.$$
By the above the groups $E^2_{2,1},E^3_{1,2},E^4_{0,3}$ are all zero. Therefore the images of the boundary maps $d^2_{2,1}:E^2_{2,1}\rightarrow E^2_{4,0},
d^3_{1,2}:E^3_{1,2}\rightarrow E^3_{4,0},d^4_{0,3}:E^4_{0,3}\rightarrow E^4_{4,0}$ are also zero. So $E^{\infty}_{4,0}$ is the subgroup of $E^2_{4,0}$ and hence torsion-free. Because $E^2_{0,3}$ is also zero, for the same reason we get that $E^{\infty}_{2,2}$ is the subgroup of $E^2_{2,2}$ and hence torsion-free, too. By part $(b)$ of Lemma 6.1 the group $E^2_{0,4}=H^0(\mathbb P^2_{\overline k},R^4\overline{\rho}_*(\mathbb Z_l))$ is torsion-free. The group $E^{\infty}_{0,4}$ is a subgroup of $E^2_{0,4}$ so it is also torsion-free. We get that $H^4(\overline Y,\mathbb Z_l)$ has a filtration by torsion-free groups, so it is torsion-free, too. 
\end{proof}
We will continue to use the notation of Lemma 4.1. Let $\rho_1:Y_1\rightarrow\mathbb P^2_k$ be the family
which we get by blowing up the isolated singularities of $Y_{xy}\subset Y$ inside the family $\rho:Y\rightarrow\mathbb P^2_k$.
\begin{prop} The following holds:
\begin{enumerate}
\item[$(a)$] we have
$$H^i(\overline Y_1,\mathbb Z_l)\cong\left\{\begin{array}{ll}
\mathbb Z_l,&\text{if $i=0$,}\\
0,&\text{if $i=1,3$,}\\
\mathbb Z_l^{122},&\text{if $i=2$.}\end{array}\right.$$
\item[$(b)$] the group $H^4(\overline Y_1,\mathbb Z_l)$ is torsion-free.
\end{enumerate}
\end{prop}
\begin{proof} First note that for every point $d\in D$ the blow-up of the unique isolated singularity of the fibre $\rho^{-1}(d)$ is isomorphic to the four-dimensional non-singular quadric hypersurface $C_5$:
$$y_0^2+y^2_1+y^2_2+y^2_3+u_1u_2=0$$
where $[y_0:y_1:y_2:y_3:u_1:u_2]$ are the homogeneous coordinates for $\mathbb P^5_{\overline k}$ and we continue to use the notation of the proof of Lemma 4.1. By Theorem 1.13 of \cite{Re1} we have:
$$H^i(C_5,\mathbb Z_l)\cong\left\{\begin{array}{ll}\mathbb Z_l,&\text{if $i=0,2,6,8$,}\\
\mathbb Z_l^2,&\text{if $i=4$,}
\\0,&\text{otherwise.}\end{array}\right.$$
Let $\sigma:Y_1\rightarrow Y$ be the blow-up map. By the proper base change theorem $R^0\overline{\sigma}_*(\mathbb Z_l)\cong
\mathbb Z_l$ and $R^1\overline{\sigma}_*(\mathbb Z_l)\cong R^3\rho_*(\mathbb Z_l)\cong0$, while $R^2\overline{\sigma}_*(\mathbb Z_l)$ and $R^4\overline{\sigma}_*(\mathbb Z_l)$ are skyscraper sheaves supported on the set $\widetilde D$ of singular points in $\rho^{-1}(D)$ such that their stalk at every $d\in\widetilde D$ is isomorphic to $\mathbb Z_l$ and $\mathbb Z_l^2$, respectively. There is a Leray-Serre spectral sequence:
$$E^2_{i,j}=H^i(\overline Y,R^j\overline{\sigma}_*(\mathbb Z_l))
\Rightarrow
H^{i+j}(\overline Y_1,\mathbb Z_l).$$
By the above
$$E^2_{i,j}\cong\left\{\begin{array}{ll}
H^i(\overline Y,\mathbb Z_l),&\text{if $j=0$,}\\
H^j(C_5,\mathbb Z_l)^{\oplus120},&\text{if $i=0,j>0$,}\\
0,&\text{otherwise,}\end{array}\right.$$
since $|\widetilde D|=|D|=120$. Clearly $H^0(\overline Y_1,\mathbb Z_l)=\mathbb Z_l$ since $Y_1$ is connected. By Proposition 6.2 and the above it is also clear that $E^2_{i,j}=0$ when $i+j=1$ or $i+j=3$, and hence $H^1(\overline Y_1,\mathbb Z_l)=H^3(\overline Y_1,\mathbb Z_l)=0$. Because the domain and the range, respectively, of the boundary maps $d^2_{0,1}:E^2_{0,1}\rightarrow E^2_{2,0}$ and $d^2_{0,2}:E^2_{0,2}\rightarrow E^2_{2,1}$ are zero, we get that
$$E^{\infty}_{2,0}\cong E^2_{2,0}\cong\mathbb Z_l^2\textrm{ and }
E^{\infty}_{0,2}\cong E^2_{0,2}\cong\mathbb Z_l^{120}.$$
Therefore we have $H^2(\overline Y_1,\mathbb Z_l)\cong\mathbb Z_l^{122}$. By the above the groups $E^2_{2,1},E^3_{1,2},E^4_{0,3}$ are all zero. Therefore the images of the boundary maps $d^2_{2,1}:E^2_{2,1}\rightarrow E^2_{4,0},
d^3_{1,2}:E^3_{1,2}\rightarrow E^3_{4,0},d^4_{0,3}:E^4_{0,3}\rightarrow E^4_{4,0}$ are also zero. So $E^{\infty}_{4,0}\cong E^2_{4,0}=H^4(\overline Y,\mathbb Z_l)$ is torsion-free by part $(b)$ of Proposition 6.2. By the above $E^2_{0,4}$ is torsion-free. The group $E^{\infty}_{0,4}$ is a subgroup of $E^2_{0,4}$ so it is also torsion-free. We get that $H^4(\overline Y_1,\mathbb Z_l)$ has a filtration by torsion-free groups, so it is torsion-free, too. 
\end{proof}
\begin{defn} We say that a morphism $\phi:A\rightarrow B$ of $n$-dimensional varieties over $\overline k$ is a simple $q$-map if it is the blow-up of a subvariety $K\subset B$ which is isomorphic to $\mathbb P^1_{\overline k}$ and the map $\phi|_{\phi^{-1}(K)}:\phi^{-1}(K)\rightarrow K$ is a quadratic fibration, that is, there is a $\mathbb P^{n-1}$-bundle $\sigma:T\rightarrow K$ and a closed immersion of $K$-schemes $\iota:\phi^{-1}(K)\rightarrow T$ such that $\phi|_{\phi^{-1}(K)}=\sigma\circ\iota$ and for every $\overline k$-valued point $p$ of $K$ the restriction $\iota|_{\phi^{-1}(p)}:\phi^{-1}(p)\rightarrow\sigma^{-1}(p)\cong\mathbb P^{n-1}_{\overline k}$ is the closed imbedding of a quadratic hypersuface into its ambient projective space. We say that a morphism $\phi:A\rightarrow B$ as above is a $q$-map if it is the composition of finitely many simple $q$-maps.
\end{defn}
\begin{prop} Let $\phi:A\rightarrow B$ be a $q$-map of $5$-dimensional projective varieties over $\overline k$. Assume that
$$H^1(B,\mathbb Z_l)=H^3(B,\mathbb Z_l)=0,$$
and both $H^2(B,\mathbb Z_l)$ and $H^4(B,\mathbb Z_l)$ are torsion-free. Then the same holds for $A$.
\end{prop}
\begin{proof} By induction we may reduce to the case when $\phi$ is a simple $q$-map. Let $\kappa:K\rightarrow B$ be the subvariety which is the centre of the blow-up $\phi$. By part $(a)$ of Propositions 5.1, 5.3, 5.5 and 5.6 and the proper base change theorem $R^0\phi_*(\mathbb Z_l)\cong
\mathbb Z_l$ and $R^1\phi_*(\mathbb Z_l)\cong R^3\phi_*(\mathbb Z_l)\cong0$, while $R^4\phi_*(\mathbb Z_l)$ is a sheaf supported on $K$ with torsion-free stalks. Let $\sigma:T\rightarrow K$ be a $\mathbb P^4$-bundle and let $\iota:\phi^{-1}(K)\rightarrow T$ be a closed immersion of $K$-schemes which satisfies the conditions of Definition 6.4 above. The sheaf $R^2\sigma_*(\mathbb Z_l)$ is lisse of rank one. Because the \'etale fundamental group of $K\cong\mathbb P^1_{\overline k}$ is trivial we get that $R^2\sigma_*(\mathbb Z_l)\cong\mathbb Z_l$. By part $(b)$ of Propositions 5.1, 5.3, 5.5 and 5.6 and the proper base change theorem the map $\iota^*:R^2\sigma_*(\mathbb Z_l)\rightarrow R^2\phi_*(\mathbb Z_l)$ is an isomorphism. Therefore $R^2\phi_*(\mathbb Z_l)\cong\mathbb Z_l$. There is a Leray-Serre spectral sequence:
$$E^2_{i,j}=H^i(B,R^j\phi_*(\mathbb Z_l))
\Rightarrow
H^{i+j}(A,\mathbb Z_l).$$
By the above we have 
$$E^2_{1,0}=H^1(B,\mathbb Z_l)=0\text{ and }E^2_{0,1}=H^0(B,0)=0,$$
and hence $H^1(A,\mathbb Z_l)=0$. Similarly by the above we have
$$E^2_{3,0}=H^3(B,\mathbb Z_l)=0,\ E^2_{2,1}=H^2(B,0)=0,$$
$$E^2_{1,2}=H^1(B,\kappa_*(\mathbb Z_l))=H^1(K,\mathbb Z_l)=0,\text{ and }E^2_{0,3}=H^0(B,0)=0,$$
and hence $H^3(A,\mathbb Z_l)=0$. Moreover
$$E^2_{2,0}=H^2(B,\mathbb Z_l),\ E^2_{1,1}=H^1(B,0)=0,\ E^2_{0,2}=H^0(B,\kappa_*(\mathbb Z_l))=H^0(K,\mathbb Z_l)=\mathbb Z_l.$$
Moreover the domain and the range, respectively, of the boundary maps $d^2_{0,1}:E^2_{0,1}\rightarrow E^2_{2,0}$ and $d^2_{0,2}:E^2_{0,2}\rightarrow E^2_{2,1}$ are zero, so we get that
$E^{\infty}_{2,0}\cong E^2_{2,0}$ and $E^{\infty}_{0,2}\cong E^2_{0,2}$,
and hence $H^2(A,\mathbb Z_l)$ has a filtration by torsion-free groups. Therefore it is torsion-free, too. We have:
$$E^2_{4,0}=H^4(B,\mathbb Z_l),\ E^2_{3,1}=H^3(B,0)=0,$$
$$E^2_{2,2}=H^2(B,\kappa_*(\mathbb Z_l))=H^2(K,\mathbb Z_l)=\mathbb Z_l,\text{ and }E^2_{1,3}=H^1(B,0)=0.$$
By the above the groups $E^2_{2,1},E^3_{1,2},E^4_{0,3}$ are all zero. Therefore the images of the boundary maps $d^2_{2,1}:E^2_{2,1}\rightarrow E^2_{4,0},
d^3_{1,2}:E^3_{1,2}\rightarrow E^3_{4,0},d^4_{0,3}:E^4_{0,3}\rightarrow E^4_{4,0}$ are also zero. So $E^{\infty}_{4,0}$ is the subgroup of $E^2_{4,0}$ and hence torsion-free. Because $E^2_{0,3}$ is also zero, for the same reason we get that $E^{\infty}_{2,2}$ is the subgroup of $E^2_{2,2}$ and hence torsion-free, too. Clearly $E^2_{0,4}=H^0(B,R^4\phi_*(\mathbb Z_l))$ is torsion-free, so its subgroup $E^{\infty}_{0,4}$ is also torsion-free. We get that $H^4(A,\mathbb Z_l)$ has a filtration by torsion-free groups, so it is torsion-free, too. 
\end{proof}
\begin{proof}[Proof of Theorem 1.3] We will continue to use the notation of Definition 4.4. Let $\phi:Y_0\rightarrow Y_1$ be the blow-up of the inverse image $\widetilde W\subset Y_1$ of $W$ with respect to the blow-up map $\sigma:Y_1\rightarrow Y$ of the proof of Proposition 6.3. Let $\psi:X\rightarrow Y_0$ be the blow-up of the proper transform $\widetilde Z\subset Y_0$ of the closed subscheme $Z\subset Y$ inside $Y_0$. By Propositions 6.3 and 6.5 it will be enough to show that both $\overline{\phi}_0$ and $\overline{\phi}_1$ are $q$-maps. Let $\tau:T\rightarrow\mathbb P^2_k$ be a $\mathbb P^4$-bundle and let $\iota:Y\rightarrow T$ be a closed immersion of $\mathbb P^2_k$-schemes such that $\rho=\tau\circ\iota$ as in the proof of Lemma 6.1, and whose description was given at the beginning of the fourth section. Let $\alpha:T_0\rightarrow T$ be the blow-up the disjoint closed subschemes $\iota(W)\subset T$ and the isolated singularities of $\iota(Y_{xy})$ inside $T$. Then there is a closed immersion $\iota_0:Y_0\rightarrow T_0$ of $\mathbb P^2_k$-schemes such that $\alpha\circ\iota_0=\iota\circ\sigma\circ\phi$. Let $\beta:S\rightarrow T_0$ be the blow-up of the proper transform $\widehat Z\subset Y_0$ of the closed subscheme $\iota(Z)\subset T$. There is a closed immersion $\kappa:X\rightarrow S$ of $\mathbb P^2_k$-schemes such that $\beta\circ\kappa=\iota_0\circ\psi$. We have the following commutative diagrams:
$$\begin{CD}
\phi^{-1}(\widetilde W)@>\iota_0>>\alpha^{-1}(\iota(W))
@.\quad\quad\psi^{-1}(\widetilde Z)@>\kappa>>\beta^{-1}(\widehat Z)\\
@V{\phi}VV@V{\alpha}VV@V{\psi}VV@V{\beta}VV\\
\widetilde W@>\iota\circ\sigma>>\iota(W)
@.\ \widetilde Z@>\iota_0>>\widehat Z.
\end{CD}$$
The upper horizontal maps are closed immersions, the lower horizontal maps are isomorphisms, while the both right vertical maps are $\mathbb P^4$-bundles. Since the irreducible components of the base change of $\widetilde W$ and $\widetilde Z$ to $\overline k$ are projective lines, it will be enough to show that the geometric fibres of the maps $\phi:\phi^{-1}(\widetilde W)\rightarrow\widetilde W$ and $\psi:\psi^{-1}(\widetilde Z)\rightarrow\widetilde Z$ are defined by quadratic equations inside the families $\alpha:\alpha^{-1}(\iota(W))
\rightarrow\iota(W)$ and $\beta:\beta^{-1}(\widehat Z)\rightarrow\widehat Z$, respectively. 

We will first describe the family over $\widetilde W$. We are going to use the notation of the proof of Lemma 4.3. Let $[r_0:r_1]$ be the homogeneous coordinates for $\widetilde W\cong\mathbb P^1_{\overline k}$. The family $\phi:\phi^{-1}(W)\rightarrow W$ is given by the equations:
$$r_1v_1v_2+y^2_0+y^2_1+y^2_4=0,$$
$$r_0v_1v_2+y^2_0+y^2_1+y^2_4=0,$$
on the open subschemes $\{r_0\neq0\}$ and $\{r_1\neq0\}$, respectively. According to the computations of Lemma 4.2 for every $p\in\widetilde Z$ the fibre $\psi^{-1}(p)$ is given by the equation
$$s_0v_1+y^2_0+y^2_1+y^2_4=0,$$
$$s_0v_1+y^2_0+y^2_1=0,$$
when $\phi(\rho(p))\in V_{xy}-\{[0:0:1]\}-\bigcup_{l\in K}l$ and when $\phi(\rho(p))\in V_{xy}\cap\bigcup_{l\in K}l$, respectively. Similarly for every $p\in\widetilde Z$ such that $\phi(\rho(p))=[0:0:1]$ the fibre $\psi^{-1}(p)$ is given by the equation
$$s_1w_2+z^2_0+z^2_1+z^2_4=0$$
according to the computations of Lemma 4.3.
\end{proof}
We finish this paper by showing that there are smooth $5$-dimensional projective rational varieties over every algebraically closed field of odd characteristic which have $2$-torsion in its fourth dimensional \'etale cohomology. In the rest of the article let $k$ denote an algebraically closed field of odd characteristic.
\begin{lemma} For every $n\in\mathbb N$ there is a smooth projective surface $X$ over $k$ such that the order of the $2$-torsion in $H^2(X,\mathbb Z_2)$ is at least $2^n$. 
\end{lemma}
\begin{proof} Let $S$ be a smooth projective variety over $k$ such that the $2$-torsion of $H^2(S,\mathbb Z_2)$ is non-trivial. (For example we may take $S$ to be the hyperelliptic surface $E\times E/(\mathbb Z/2\mathbb Z)$ where $E$ is an elliptic curve and the generator of $\mathbb Z/2\mathbb Z$ acts on $E\times E$ by translation by a non-zero $2$-torsion element on the first factor and by multiplication by $-1$ on the second factor.) Let $S^n$ denote the $n$-fold self-product of $S$. By the K\"unneth formula the order of the $2$-torsion in $H^2(S^n,\mathbb Z_2)$ is at least $2^n$. By the Lefschetz hyperplane section theorem there is a smooth projective surface $X\subseteq S^n$ such that the pull-back map $H^2(S^n,\mathbb Z_2)\rightarrow H^2(X,\mathbb Z_2)$ is injective. The claim is now clear.
\end{proof}
\begin{prop} There is a smooth $5$-dimensional projective rational variety $Y$ over $k$ such that the order of the $2$-torsion in $H^2(Y,\mathbb Z_2)$ is non-zero.
\end{prop}
\begin{proof} Let $X$ be a smooth projective surface over $k$ such that the order of the $2$-torsion in $H^2(X,\mathbb Z_2)$ is at least $4$. As it is very well-known there is a closed imbedding $i:X\rightarrow\mathbb P^5_k$ (see for example Theorem 2.25 of \cite{Sh1} on page 134). Let $Y$ be the variety which we get by blowing up $i(X)$ inside $\mathbb P^5_k$ and let $j:Y\rightarrow\mathbb P^5_k$ be the blow-up map. Then $j:j^{-1}(i(X))\rightarrow i(X)$ is a $\mathbb P^2$-bundle. Because the fibres of $j$ are connected we have $R^0j_*(\mathbb Z_2)\cong\mathbb Z_2$ by the proper base change theorem. Moreover the sheaf $R^ij_*(\mathbb Z_2)$ is zero when $i=1,3$, or at least $5$, again by the proper base change theorem. For $i=2$ or $4$ there is a lisse sheaf $\mathcal F_i$ of rank one on $X$ such that $R^ij_*(\mathbb Z_2)\cong i_*(\mathcal F_i)$. The chern class of the canonical bundle of the fibres of the map $j:j^{-1}(i(X))\rightarrow i(X)$ furnishes a non-zero section of $\mathcal F_2$ and hence $\mathcal F_2\cong\mathbb Z_2$. The cup-product of this section with itself gives a a non-zero section of $\mathcal F_4$ and hence $\mathcal F_4\cong\mathbb Z_2$, too. Therefore $R^2j_*(\mathbb Z_2)\cong R^4j_*(\mathbb Z_2)\cong\mathbb Z_2$. There is a Leray-Serre spectral sequence:
$$E^2_{p,q}=H^p(\mathbb P^5_k,R^qj_*(\mathbb Z_2))
\Rightarrow
H^{p+q}(Y,\mathbb Z_2).$$
By the above
$$E^2_{4,0}=H^4(\mathbb P^5_k,\mathbb Z_2)=\mathbb Z_2,\ E^2_{3,1}=H^3(\mathbb P^5_k,0)=0,$$
$$E^2_{2,2}=H^2(\mathbb P^5_k,i_*(\mathbb Z_2))=H^2(X,\mathbb Z_2),\ E^2_{1,3}=H^1(\mathbb P^5_k,0)=0,\text{ and }$$
$$E^2_{4,0}=H^4(\mathbb P^5_k,i_*(\mathbb Z_2))=H^4(X,\mathbb Z_2)=\mathbb Z_2.$$
As the groups $E^2_{0,3},E^2_{4,1},E^3_{5,0}$ are zero, the boundary maps $d^2_{2,1}:E^2_{0,3}\rightarrow E^2_{2,2},
d^2_{2,2}:E^2_{2,2}\rightarrow E^2_{4,1},d^3_{2,2}:E^3_{2,2}\rightarrow E^0_{5,0}$ are zero maps, and hence $E^{\infty}_{4,0}\cong E^2_{4,0}$. Similarly $E^2_{2,1},E^3_{1,2},E^4_{0,3}$ are zero, and hence $E^{\infty}_{4,0}\cong E^2_{4,0}$. Therefore there is a subgroup $G\leq H^4(Y,\mathbb Z_2)$ which sits in a short exact sequence:
$$\begin{CD} 0@>>>\mathbb Z_2@>>>
G@>>>H^2(X,\mathbb Z_2)@>>>0\end{CD}.$$
Since there are at least two linearly independent $2$-torsion elements in $H^2(X,\mathbb Z_2)$, the group $G$ has non-trivial $2$-torsion.
\end{proof}

\end{document}